\documentclass[11pt]{amsart}
\usepackage{mathrsfs}
\usepackage{geometry}
\usepackage{graphicx}
\usepackage{amssymb}
\usepackage{epstopdf}
\usepackage{amsmath,amscd}
\usepackage{amsthm}
\usepackage{enumerate}
\usepackage{url,verbatim}
\usepackage{subfig}
\usepackage{enumitem}
\usepackage{tikz}
\usepackage{enumerate}

\usepackage[normalem]{ulem}
\usepackage{fixme}

\calclayout

\makeatletter
\newcommand\xleftrightarrow[2][]{%
  \ext@arrow 9999{\longleftrightarrowfill@}{#1}{#2}}
\newcommand\longleftrightarrowfill@{%
  \arrowfill@\leftarrow\relbar\rightarrow}
\makeatother

\RequirePackage[colorlinks,citecolor=blue,urlcolor=blue]{hyperref}
\usepackage{breakurl}
\theoremstyle{plain}
\newtheorem{theorem}{Theorem}
\newtheorem{definition}[theorem]{Definition}
\newtheorem{lemma}[theorem]{Lemma}

\newtheorem{question}[theorem]{Question}

\newcommand\NN{{\mathbb N}}

\newcommand\PP{{\mathbb P}}

\renewcommand\ell{l}

\newcounter{mycount}

\numberwithin{equation}{section}
\numberwithin{theorem}{section}
\numberwithin{figure}{section}

\title{Critical site percolation and cutsets}

\author{Zhongyang Li}

\begin{document}
\maketitle

\begin{abstract}In 2003, Kahn conjectured a characterization of the critical percolation probability $p_c$
  in terms of vertex cut sets (\cite{JK03}). Later, Lyons and Peres (2016) conjectured a similar characterization of $p_c$
 , but in terms of edge cut sets (\cite{LP16}). Both conjectures were subsequently proven by Tang (\cite{pt23}) for bond percolation and site percolation on bounded-degree graphs. Tang further conjectured that Kahn's vertex-cut characterization for 
$p_c^{site}$
  and the Lyons-Peres edge-cut characterization for 
$p_c^{site}$ would hold for site percolation on any infinite, connected, locally finite graph.

In this paper, we establish Kahn's vertex-cut characterization for 
$p_c^{site}$ by adapting arguments from \cite{jmh57a, DCT15}. Additionally, we disprove the Lyons-Peres edge-cut characterization for $p_c^{site}$
by constructing a counterexample.
\end{abstract}

\section{Introduction}

Introduced by Broadbent and Hammersley in 1957 (see \cite{BH57}) to study the random spread of a fluid through a medium, percolation has been a celebrated model illustrating the phase transition, magnetization, or the spread of pandemic diseases; see \cite{grgP,LP16,HDC18} for recent accounts of the theory.

Let $G=(V,E)$ be a graph. We write $e=\langle u,v\rangle$ for an edge with endpoints $u$ and $v$; where $u,v\in V$ and $e\in E$.
The (\emph{vertex}-)\emph{degree} of a vertex $v\in V$ is the number 
of edges incident to $v$; i.e.~edges one of whose endpoints is $v$.  We say a graph is locally finite if each vertex has finite degree.

 Assume $G=(V,E)$ is an infinite, locally finite, connected graph. 
A site percolation configuration $\omega\in \{0,1\}^{V}$  is a an assignment to each vertex in $G$ of either state 0 or state 1. A cluster in $\omega$ is a maximal connected set of vertices in which each vertex has the same state in $\omega$. A cluster may be a 0-cluster or a 1-cluster depending on the common state of vertices in the cluster.  A cluster may be finite or infinite depending on the total number of vertices in it. We say that percolation occurs in $\omega$ if there exists an infinite 1-cluster in $\omega$. 

Of particular interest is the i.i.d.~Bernoulli site percolation on a graph. In such a model, an independent Bernoulli random variable, which takes value 1 with probability $p\in [0,1]$, is associated to each vertex. We use $\PP_p$ to denote the probability measure for i.i.d.~Bernoulli($p$) site percolation on $G$.

For the i.i.d.~Bernoulli site percolation, define
\begin{small}
\begin{align}
p_c^{site}(G):&=&\inf\{p\in[0,1]: \mathrm{Bernoulli}(p)\ \mathrm{site\ percolation\ on}\ G\ \mathrm{has\ an\ infinite\ 1-cluster\ a.s.} \}\label{dpc}
\end{align}
\end{small}

Determining the exact value of 
$p_c$
  has attracted considerable interest among probabilists. Although the value of 
$p_c$ is known for certain specific graphs (see e.g. \cite{HK80,GM12}), it remains unknown for the vast majority of general graphs, including many transitive graphs. Significant efforts have been made to find exact characterizations of 
$p_c$, such as descriptions of 
$p_c$ based on other properties of the graph. For example, it was proved in \cite{ry90} that for an infinite tree, $p_c$ is equal to the reciprocal of the average branching number of the tree.

In 2003, Kahn conjectured a characterization of critical percolation probability $p_c$ in terms of vertex  cut sets (\cite{JK03}); followed by a conjecture of a characterization of critical percolation probability in terms of edge cut sets by Lyons and Peres in 2016 (\cite{LP16}). Both conjectures were proved by Tang (\cite{pt23}) for bond percolation and site percolation on bounded degree graphs. Tang also conjectured that both Kahn's vertex-cut characterization for $p_c^{site}$ and Lyons-Peres edge-cut characterization for $p_c^{site}$ hold for site percolation on any infinite connected locally finite graphs. Indeed, site percolation is more general than bond percolation in the sense that the Bernoulli bond percolation on every graph is Bernoulli site percolation on a different graph but not vice versa. In this paper, we prove the Kahn's vertex-cut characterization for $p_c^{site}$ and disprove Lyons-Peres edge-cut characterization for $p_c^{site}$.

In a powerful refinement of a classical argument by Hammersley \cite{jmh57a}, Duminil-Copin and Tassion \cite{DCT15} showed that for transitive graphs 
$G$, the critical value 
$p_c(G)$ can be characterized in terms of the mean number of points on the surface of a box that are connected to its root. This paper extends that work in \cite{jmh57a,DCT15} to general locally finite graphs without assuming transitivity or quasi-transitivity, using the technique of differential inequalities (\cite{AB87}); based on which we prove Kahn's conjecture.

One generalization of the arguments in \cite{jmh57a,DCT15} was obtained for site percolation on infinite, connected, bounded degree graphs in \cite{pt23}; the major difference from \cite{pt23} is our technique here does not require the graph to have bounded degree, but only require the graph to be locally finite. Although \cite{DCT15} was written in the context of bond percolation, Sect. 1.2 of \cite{DCT15} mentions that the proof may be adapted to site percolation on transitive graphs and suggests a possible definition of $\varphi_p(S)$ for site percolation. The paper \cite{pt23} adopts this suggested definition, which imposes a bounded degree constraint. In this paper, we introduce a different definition of $\varphi_p(S)$: instead of counting the expected number of boundary vertices connected to the center, as in \cite{DCT15, pt23}, we count the expected number of boundary vertices that have a connected neighbor to the center. This new definition allows us to remove the bounded degree constraint. This generalization to locally finite graphs was also used to prove the existence of infinitely many infinite clusters on planar graphs with vertex degree at least 7 (see \cite{ZL23,ZL232}), resolving a conjecture of Benjamini and Schramm in 1996 (\cite{bs96}).

For the Lyons-Peres edge-cut characterization for $p_c^{site}$, it appears that the bounded degree assumption for the underlying graph is critical. Indeed, we shall show that this characterization for $p_c^{site}$ does not hold for all the infinite connected locally finite graphs by constructing a counter example.

The organization of the paper is as follows: in Section \ref{sect:vc}, we prove the Kahn's vertex-cut characterization for $p_c^{site}$ on all infinite connected, locally finite graphs (Theorem \ref{t03}).
In Section \ref{sect:ec}, we prove that the Lyons-Peres edge-cut characterization for $p_c^{site}$ does not hold all infinite connected, locally finite graphs (Theorem \ref{t34}).

\section{Vertex Cut}\label{sect:vc}

In this section, we prove the Kahn's vertex-cut characterization for $p_c^{site}$ on all infinite connected, locally finite graphs.
The proof is based on a generalization of the arguments in \cite{jmh57a,DCT15} to site percolation on infinite connected, locally finite graphs. The main theorem proved in this section is Theorem \ref{t03}.

\begin{definition}\label{df21}Suppose $G=(V,E)$ is an infinite, connected locally finite graph. For $x\in V$, a vertex cutset $\Pi_V\subset V$ separating $x$ from $\infty$ is a subset of vertices such that any singly infinite path starting from $x$ must occupy at least one vertex in $\Pi_{V}$.

Let $\Pi_V$ be a vertex cut set which separates $x$ from $\infty$. For each $v\in \Pi_V$, let $A(x,v,\Pi_V)$ denote the event that $x$ is connected to $v$ by an open path without using any vertices in $\Pi_V\setminus \{v\}$. Define
\begin{align*}
p'_{cut,V}:=p'_{cut,V}(G):=\sup\left\{p\geq 0:\inf_{\Pi_V}\sum_{v\in \Pi_V}\mathbb{P}_p[A(x,v,\Pi_V)]=0\right\}.
\end{align*}
where the infimum is taken over all vertex cutsets $\Pi_V$ separating $x$ from $\infty$.
\end{definition}

Note that the definition of $p'_{cut,V}$ above does not depend on the choice of the vertex $x\in V(G)$.

\begin{question}\label{q01}(Question 3.1 in \cite{JK03})Does $p_c(G)=p_{cut,V}'(G)$ hold for all $G$?
\end{question}

Question \ref{q01} was answered positively in \cite{pt23} for Bernoulli bond percolation and for Bernoulli site percolation with bounded degree. Here we answer this question positively for Bernoulli site percolation without the bounded degree assumption.

\begin{theorem}\label{t03}For Bernoulli site percolation on every locally finite, connected, infinite graph $G$, one has that
\begin{align*}
p_{cut,V}'=p_c^{site}.
\end{align*}
\end{theorem}

\begin{lemma}\label{l04}For Bernoulli site percolation on every locally finite, connected, infinite graph $G$,
\begin{align}
p_{cut,V}'\leq p_c^{site}\label{llb}
\end{align}
\end{lemma}

\begin{proof}Note that
\begin{align*}
\mathbb{P}_p(x\leftrightarrow\infty)\leq\inf\sum_{v\in \Pi_V}\mathbb{P}_p(A(x,v,\Pi_E))
\end{align*}

Then (\ref{llb}) follows from the definitions of $p_{cut,V}'$,  and $p_c^{site}$.
\end{proof}

Before proving Theorem \ref{t03}, we recall a few lemmas proved in \cite{ZL23}.

Let $G=(V,E)$ be a graph. For each $p\in (0,1)$, let $\mathbb{P}_p$ be the probability measure of the i.i.d.~Bernoulli($p$) site percolation on $G$.
For each $S\subset V$, let $S^{\circ}$ consist of all the interior vertices of $S$, i.e., vertices all of whose neighbors are in $S$ as well.
For each  $S\subseteq V$, $v\in S$, define
\begin{align*}
\varphi_p^{v}(S):=\begin{cases}\sum_{y\in S:[\partial_V y]\cap S^c\neq\emptyset}\mathbb{P}_p(v\xleftrightarrow{S^{\circ}} \partial_V y)&\mathrm{if}\ v\in S^{\circ}\\
1&\mathrm{if}\ v\in S\setminus S^{\circ}
\end{cases}
\end{align*}
where 
\begin{itemize}
\item $v\xleftrightarrow{S^{\circ}} x$ is the event that the vertex $v$ is joined to the vertex $x$ by an open path visiting only interior vertices in $S$;
\item let $A\subseteq V$; $v\xleftrightarrow{S^{\circ}} A$ if and only if there exists $x\in A$ such that $v\xleftrightarrow{S^{\circ}} x$;
\item $\partial_V y$ consists of all the vertices adjacent to $y$.
\end{itemize}

\begin{lemma}\label{l71}Let $G=(V,E)$ be an infinite, connected, locally finite graph. The critical site percolation probability on $G$ is given by
\begin{align*}
    \tilde{p}_c=\sup\{p\geq 0:\exists \epsilon_0>0, \mathrm{s.t.}\forall v\in V, \exists S_v\subseteq V\ \mathrm{satisfying}\ |S_v|<\infty\ \mathrm{and}\ v\in S_v^{\circ}, \varphi_p^{v}(S_v)\leq 1-\epsilon_0\}
\end{align*}
Moreover
\begin{enumerate}
    \item If $p>\tilde{p}_c$, a.s.~there exists an infinite 1-cluster; moreover, for any $\epsilon>0$ there exists a vertex $w$, satisfying 
    \begin{align}
    \mathbb{P}_p(w\leftrightarrow \infty)\geq 1-\left(\frac{1-p}{1-\tilde{p}_c}\right)^{1-\epsilon}\label{lbc}
    \end{align}
    \item If $p<\tilde{p}_c$, then for any vertex $v\in V$
    \begin{align}
        \mathbb{P}_p(v\leftrightarrow \infty)=0.\label{lbc2}
    \end{align}
\end{enumerate}
In particular, (1) and (2) implies that $p_c^{site}(G)=\tilde{p}_c$
\end{lemma}

\begin{lemma}\label{l72}Let $G=(V,E)$ be an infinite, connected, locally finite graph. For all $p>0$, $\Lambda\subset V$ finite and $v\in \Lambda$ 
\begin{align*}
\frac{d}{dp}\mathbb{P}_p(v\leftrightarrow \Lambda^c)\geq \frac{1}{1-p}
\left[ \mathrm{inf}_{S:v\in S,|S|<\infty}\varphi_p^v(S)\right]
\left(1-\mathbb{P}_p(v\leftrightarrow \Lambda^{c})\right)
\end{align*}
\end{lemma}

\begin{proof}Define the random subset $\mathcal{H}$ of $\Lambda$:
\begin{align*}
    \mathcal{H}:=\{x\in \Lambda.\ \mathrm{s.t.}\ x\nleftrightarrow \Lambda^c\}
\end{align*}
Recall that for each $y\in V$, if $\omega^y\in \{v\leftrightarrow\Lambda^c\}$ and $\omega_y\notin \{v\leftrightarrow\Lambda^c\}$, then $y$ is pivotal for the event $v\leftrightarrow\Lambda^c$. Here for each $\omega\in \{0,1\}^{V}$,
\begin{align*}
    \omega^y(u):=\begin{cases}\omega(u);& \mathrm{If}\ u\neq y;\\ 1&\mathrm{If}\ u= y.\end{cases}\qquad
\omega_y(u):=\begin{cases}\omega(u);& \mathrm{If}\ u\neq y;\\ 0&\mathrm{If}\ u=y.\end{cases}
\end{align*}
By Russo's formula (\cite{RL81})
\begin{align}
  \frac{d}{dp}\mathbb{P}_p[v\leftrightarrow\Lambda^c]&=\sum_{y\in V}\mathbb{P}_p(y\ \mathrm{is\ pivotal\ for\ }v\leftrightarrow\Lambda^c)\label{pp1}\\
   &\geq\frac{1}{1-p} \sum_{y\in V}\mathbb{P}_p(y\ \mathrm{is\ pivotal\ for\ }v\leftrightarrow\Lambda^c\ \mathrm{and}\ v\nleftrightarrow\Lambda^c)\notag\\
   &\geq \frac{1}{1-p}\sum_{S:v\in S}\sum_{y\in V}\mathbb{P}_p(y\ \mathrm{is\ pivotal\ for\ }v\leftrightarrow\Lambda^c\ \mathrm{and}\ \mathcal{H}=S)\notag
\end{align}
Note that if $v\in S^{\circ}$
the event that $y$ is pivotal for $v\leftrightarrow\Lambda^c$ and $\mathcal{H}=S$ is nonempty only if 
\begin{itemize}
\item $y\in S$ and $y$ is adjacent to a vertex $x\notin S$; and
\item $y$ is closed; and
\item $y$ is adjacent to a vertex $z$, s.t. $z \xleftrightarrow{S^{\circ}} v$
\end{itemize}
Then for each $v\in S^{\circ}$, $y\in S$ and $\partial_V y\cap S^{c}\neq \emptyset$ we have 
\begin{align}
   \mathbb{P}_p(y\ \mathrm{is\ pivotal\ for\ }v\leftrightarrow\Lambda^c\ \mathrm{and}\ \mathcal{H}=S)
   =\mathbb{P}(\partial_V y \xleftrightarrow{S^{\circ}} v)\mathbb{P}(\mathcal{H}=S);\label{pp2}
\end{align}
since the event $\partial_V y \xleftrightarrow{S^{\circ}} v$ depends only on vertices in $S$ with all neighbors in $S$ while the event $\mathcal{H}=S$ depends only on vertices in $S$ that have at least one neighbor not in $S$.

Moreover, if $v\in S\setminus S^{\circ}$
the event that $y$ is pivotal for $v\leftrightarrow\Lambda^c$ and $\mathcal{H}=S$ is nonempty only if $y=v$. In this case for each $v\in S\setminus S^{\circ}$, $y\in S$ and $\partial_Vy \cap S^c\neq \emptyset$
\begin{align}
   \mathbb{P}_p(y\ \mathrm{is\ pivotal\ for\ }v\leftrightarrow\Lambda^c\ \mathrm{and}\ \mathcal{H}=S)
   =\mathbf{1}_{y=v}\mathbb{P}(\mathcal{H}=S);\label{pp3}
\end{align}

Plugging (\ref{pp2}) and (\ref{pp3}) into (\ref{pp1}), we obtain
\begin{align*}
 \frac{d}{dp}\mathbb{P}_p[v\leftrightarrow\Lambda^c]
 &\geq\frac{1}{1-p} \sum_{S:v\in S^{\circ}}\sum_{y\in S:\partial_V y\cap S^{c}\neq\emptyset}\mathbb{P}_p(\partial_V y \xleftrightarrow{S^{\circ}} v)\mathbb{P}_p(\mathcal{H}=S)\\
 &+\frac{1}{1-p} \sum_{S:v\in S\setminus S^{\circ}}\mathbb{P}_p(\mathcal{H}=S)\\
 &\geq\frac{1}{1-p}\left[ \mathrm{inf}_{S:v\in S,|S|<\infty}\varphi_p^v(S)\right]\sum_{S:v\in S}\mathbb{P}_p(\mathcal{H}=S).
\end{align*}
Then the lemma follows from the fact that 
\begin{align*}
    \sum_{S:v\in S}\mathbb{P}_p(\mathcal{H}=S)=\mathbb{P}_p(v\nleftrightarrow \Lambda^{c})=1-\mathbb{P}_p(v\leftrightarrow \Lambda^{c})
\end{align*}
\end{proof}

\begin{lemma}\label{l73}Let $G=(V,E)$ be an infinite, connected, locally finite graph. Let $p>0$, $u\in S\subset A$ and $B\cap S=\emptyset$. Then
\begin{itemize}
\item If $u\in S^{\circ}$
\begin{align*}
    \mathbb{P}_p(u\xleftrightarrow{A} B)\leq \sum_{y\in S:\partial_V y\cap S^c\neq \emptyset}
    \mathbb{P}_p(u\xleftrightarrow{S^{\circ}}\partial_V y )\mathbb{P}_p(y\xleftrightarrow{A} B)
\end{align*}
\item If $u\in S\setminus S^{\circ}$,
\begin{align*}
    \mathbb{P}_p(u\xleftrightarrow{A} B)\leq \sum_{y\in S:\partial_V y\cap S^c\neq \emptyset}
    \mathbf{1}_{y=u}\mathbb{P}_p(y\xleftrightarrow{A} B)
\end{align*}
\end{itemize}
\end{lemma}

\begin{proof}The conclusion is straightforward when $u\in S\setminus S^{\circ}$. It remains to prove the case when $u\in S^{\circ}$.
Let $u\in S^{\circ}$ and assume that the event $u\xleftrightarrow{A}B$ holds. Consider an open path $\{u_j\}_{0\leq j\leq K}$ from $u$ to $B$. Since $B\cap S=\emptyset$, one can find the least $k$ such that $u_{k+1}\notin S^{\circ}$.
Then the following events occur disjointly:
\begin{itemize}
    \item $u$ is connected to $\partial_V u_{k+1}$ by an open path in $S^{\circ}$;
    \item $u_{k+1}$ is connected to $B$ in $A$.
\end{itemize}
Then the lemma follows.
\end{proof}

\bigskip
\noindent\textbf{Proof of Lemma \ref{l71}(1).}
Let $p>\tilde{p}_c$. Define
\begin{align}
    f_v(p):=\mathbb{P}_p[v\leftrightarrow \Lambda^{c}]\label{dfv}
\end{align}
For any $\epsilon>0$, one can choose $\epsilon_1\in(0,\epsilon)$ and $p_1\in(\tilde{p}_c,p)$, such that
\begin{align}
\left(\frac{1-p}{1-p_1}\right)^{1-\epsilon_1}\leq \left(\frac{1-p}{1-\tilde{p}_c}\right)^{1-\epsilon}.\label{p1pc}
\end{align}
There exists $w\in V$, such that $\varphi_p^{w}(S)>1-\epsilon_1$ for all $p>p_1$. By Lemma \ref{l72},
\begin{align*}
    \frac{f_w'(p)}{1-f_w(p)}\geq \frac{1-\epsilon_1}{1-p},\ \forall p>p_1.
\end{align*}
Integrating both the left hand side and right hand side from $p_1>\tilde{p}_c$ to $p$, we infer that
\begin{align*}
\mathbb{P}_p(w\leftrightarrow\infty)\geq 1-\left(\frac{1-p}{1-p_1}\right)^{1-\epsilon_1}
\end{align*}
Then the lemma follows from (\ref{p1pc})
\hfill$\Box$

\bigskip
\noindent\textbf{Proof of Lemma \ref{l71}(2).} We shall use $p_c$ to denote $p_c^{site}(G)$ when there is no confusion. Note that Part (1) implies that $p_c\leq \tilde{p}_c$. If $p_c<\tilde{p}_c$, it suffices to show that there exists $p'\in (p_c,\tilde{p}_c)$ such that
\begin{align*}
   \mathbb{P}_{p'}(v\leftrightarrow\infty)=0,\ \forall v\in V;
\end{align*}
which contradicts the definition of $p_c$; then $p_c=\tilde{p}_c$ and (\ref{lbc2}) holds.

From the definition of $\tilde{p}_c$, we see that if $p_c<\tilde{p}_c$, there exists $p'\in (p_c,\tilde{p}_c)$ such that there exists $\epsilon_0>0$ for all $v\in V$, there exists a finite $S_v\subseteq V$ satisfying $v\in S_v^{\circ}$ and 
\begin{align*}
    \varphi_{p'}^v(S_v)\leq 1-\epsilon_0.
\end{align*}
By Lemma \ref{l73},
\begin{align*}
    \mathbb{P}_{p'}(v\leftrightarrow \infty)\leq \sum_{y\in S_v:\partial_V y\cap S_v^{c}\neq \emptyset}\mathbb{P}_{p'}(v\xleftrightarrow{S_v^{\circ}}\partial_V y)\mathbb{P}_{p'}(y\leftrightarrow\infty)
\end{align*}
Similarly, there exists a finite $S_y\subseteq V$ satisfying $y\in S_y^{\circ}$ and 
\begin{align*}
    \varphi_{p'}^y(S_y)\leq 1-\epsilon_0.
\end{align*}
Again by Lemma \ref{l73},
\begin{align*}
    \mathbb{P}_{p'}(y\leftrightarrow \infty)\leq \sum_{y_1\in S_y:\partial_V y_1\cap S_y^{c}\neq \emptyset}\mathbb{P}_{p'}(y\xleftrightarrow{S_v^{\circ}}\partial_V y_1)\mathbb{P}_{p'}(y_1\leftrightarrow\infty)
\end{align*}
Since the graph is locally finite, the process can continue for infinitely many steps. Hence we have
\begin{align*}
    \mathbb{P}_{p'}(v\leftrightarrow \infty)\leq\lim_{n\rightarrow\infty}(1-\epsilon_0)^n=0.
\end{align*}
Then the lemma follows.
\hfill$\Box$
\\
\bigskip
\\
\noindent\textbf{Proof of Theorem \ref{t03}.} By Lemma \ref{l04}, it suffices to show that whenever $p>p_{cut,V}'$, $p\geq p_c^{site}$.

By Definition \ref{df21}, if $p_1>p_{cut,V}'$, there exists $p_1>p_2>p_{cut,V}'$, such that for any $x\in V$, there exists $\epsilon>0$ ,
\begin{align*}
\inf_{\Pi_V}\sum_{v\in \Pi_V}\mathbb{P}_p[A(x,v,\Pi_V)]\geq \inf_{\Pi_V}\sum_{v\in \Pi_V}\mathbb{P}_{p_2}[A(x,v,\Pi_V)]>
\epsilon;\ \forall p>p_2.
\end{align*}
It follows that
\begin{align*}
\inf_{S:x\in S,|S|<\infty}\varphi_p^{x}(S)>\epsilon;\ \forall p>p_2.
\end{align*}
By Lemma \ref{l72}, we infer that
\begin{align*}
\mathbb{P}_p(x\leftrightarrow\infty)\geq 1-\left(\frac{1-p}{1-p_2}\right)^{\epsilon}>0.
\end{align*}
Then $p\geq p_c^{site}$, and the theorem follows.
$\hfill\Box$

\section{Edge Cut}\label{sect:ec}

In this section, we prove that the Lyons-Peres edge-cut characterization for $p_c^{site}$ does not hold all infinite connected, locally finite graphs by constructing a counterexample. The main theorem proved in this section is Theorem \ref{t34}.

\begin{definition}Suppose $G=(V,E)$ is an infinite, connected locally finite graph. An edge cutset $\Pi_E\subset E$ separating $x$ from $\infty$ is a subset of edges such that any singly infinite path starting from $x$ must occupy at least one edge in $\Pi_{E}$.

For an edge cutset $\Pi_E$ separating $x$ from $\infty$ and $e=(u,v)\in \Pi_E$, let $A(x,e,\Pi_E)$ denote the event that $x$ is connected to an endpoint of $e$ via an open path without using edges in $\Pi_E\setminus\{e\}$. Here we assume both $u$ and $v$ are open on $A(x,e,\Pi_E)$. Define
\begin{align}
p_{cut,E}'=p_{cut,E}'(G):=\sup\{p\geq 0:\inf_{\Pi_E}\sum_{e\in \Pi_E}\mathbb{P}_p(A(x,e,\Pi_E))=0\}.\label{dpe}
\end{align}
where the infimum is taken over all edge cutsets $\Pi_E$ separating $x$ from $\infty$.
\end{definition}

Note that the definition of $p'_{cut,E}$ above does not depend on the choice of the vertex $x\in V(G)$.

\begin{question}\label{q516}(Question 5.16 in \cite{LP16})Does $p_c(G)=p_{cut,E}'(G)$ hold for every locally finite, connected, infinite graph $G$?
\end{question}

\begin{lemma}\label{l23}Let $G=(V,E)$ be an infinite, connected locally finite graph.
\begin{align*}
p_{cut,E}'(G)\leq p_c^{site}(G).
\end{align*}
\end{lemma}

\begin{proof}It suffices to show that for any $p>p_c^{site}(G)$, $p\geq p_{cut,E}'(G)$. For each $v\in V$, and any edge cutset $\Pi_{E}$ separating $v$ from $\infty$,
\begin{align*}
\mathbb{P}_p(v\leftrightarrow \infty)
\leq \sum_{e\in \Pi_E}\mathbb{P}_p(A(v,e,\Pi_E))
\end{align*}
Since $\Pi_{E}$ is an arbitrary edge set separating $v$ from $\infty$, we infer that whenever $p>p_c^{site}(G)$, 
\begin{align*}
0<\mathbb{P}_p(v\leftrightarrow \infty)
\leq \inf_{\Pi_{E}}\sum_{e\in \Pi_E}\mathbb{P}_p(A(v,e,\Pi_E))
\end{align*}
It follows from (\ref{dpe}) that $p\geq p_{cut,E}'$.
\end{proof}

\begin{theorem}\label{t34}There exists a locally finite, connected, infinite graph $G$ for which $p_{cut,E}'(G)<p_c^{site}(G)$.
\end{theorem}

\begin{proof}
We shall prove the theorem by constructing a locally finite, connected, infinite graph $G$ for which $p_{cut,E}'(G)<p_c^{site}(G)$.

\begin{figure}
\includegraphics{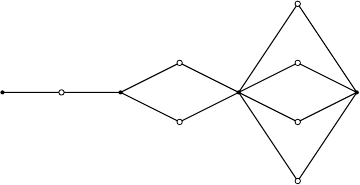}
\caption{A graph with $p_{cut,E}'(G)<p_c^{site}(G)$}
\label{fig:ecc}
\end{figure}

Let $G$ be a bipartite graph whose black vertices are located at points $(2n,0)$ $n\in \NN=\{0,1,2,\ldots\}$ in the plane. Its white vertices have coordinates $(2n+1,y_{n,i})$ for $n\in \NN$ and $1\leq i\leq 2^n$; such that
\begin{itemize}
\item  for each $n\in \NN$ and $1\leq i\leq 2^n$, $(2n,0)$ is joined to $(2n+1,y_{n,i})$ by an edge; and 
\item  for each $n\in \NN$ and $1\leq i\leq 2^n$, $(2n+2,0)$ is joined to $(2n+1,y_{n,i})$ by an edge.
\item $y_{n,1}<y_{n,2}<\ldots<y_{n,2^n}$.
\end{itemize}
See Figure \ref{fig:ecc}.

One can see that $p_c^{site}(G)=1$. However, for each edge cut set $\Pi_E$ separating $(0,0)$ from $\infty$, let
\begin{align}
N:=N(\Pi_E):=\min\{n\geq 1:\Pi_{E}\ \mathrm{separates}\ (0,0)\ \mathrm{from}\  (2n,0)\}.\label{dfN}
\end{align}
Let $x=(0,0)$. Then 
\begin{itemize}
\item Let $E_N$ be the set of all $2^N$ edges between $(2(N-1),0)$ and $(2N,0)$, let $F_N=E_N\cap \Pi_E$, then $F_N$ contains at least $2^{N-1}$ edges. To see why that is true, from (\ref{dfN}) we see that $\Pi_{E}$ separates $(2N-2,0)$ and $(2N,0)$. Each self-avoiding path in $G$ joining $(2N-2,0)$ and $(2N,0)$ must use one set of the following $2^{N-1}$ sets of edges
\begin{align}
Q_i:=\{((2N-2,0),(2N-1,y_{N-1,i})),((2N-1,y_{N-1,i}),(2N,0))\};\ 1\leq i\leq 2^{N-1}.\label{pts}
\end{align}
Hence in order to separate $(2N_2,0)$ and $(2N,0)$, $\Pi_{E}$ must contain at least one edge in each set of the $2^{N-1}$ sets in (\ref{pts}).
\item If for some $1\leq i\leq 2^{N-1}$, $F_N$ contains two edges in $Q_i$, remove the edge $(2N-1,y_{N-1,i}),(2N,0))$ from $F_N$. This way we obtain a new set of edges $\tilde{F}_N$ with exactly $2^{N-1}$ edges in it.

For each $e=(z_m,z_{m+1})\in \tilde{F}_N$, let $m$ and $m+1$ be the horizontal coordinates of $z_m$ and $z_{m+1}$, respectively. Then $m\in\{2N-2,2N-1\}$. From (\ref{dfN}) and the construction of $\tilde{F}_N$ we see that there exists a path of length at least $m$ joining $x$ and $z_m$ without using edges in $\Pi_{e}\setminus\{e\}$. It follows that
\begin{align*}
\mathbb{P}_p(A(x,e,\Pi_E))\geq p^{2N-1}
\end{align*}
\end{itemize}
Hence we have 
\begin{align*}
\sum_{e\in \Pi_E}\mathbb{P}_p(A(x,e,\Pi_E))\geq 
\sum_{e\in \tilde{F}_N}\mathbb{P}_p(A(x,e,\Pi_E))\geq
2^{N-1}p^{2N-1}
\end{align*}
which goes to infinity as $N\rightarrow\infty$ and $p>\frac{\sqrt{2}}{2}$. We infer that
\begin{align*}
\sup\left\{p\geq 0:\inf_{\Pi_E}\sum_{e\in \Pi_E}\mathbb{P}_p(A(x,e,\Pi_E))=0\right\}\leq \frac{\sqrt{2}}{2}.
\end{align*}
It follows that $p_{cut,E}'(G)<p_c^{site}(G)$.
\end{proof}
\bigskip

\noindent\textbf{Acknowledgements.} The author would like to express gratitude to Russell Lyons for highlighting Paper \cite{pt23}. The author thanks the anonymous reviewers for their careful reading of the paper and for their helpful suggestions to improve its readability.

\bibliography{percLP}
\bibliographystyle{plain}
\end{document}